\documentclass[a4paper,reqno]{amsart}
\usepackage[english]{babel}
\usepackage{amsmath, amssymb, amsthm, amscd}
\usepackage{enumerate}
\usepackage{palatino}
\usepackage{mathpazo}
\usepackage{paralist}
\usepackage[a4paper]{geometry}

\usepackage{tikz} 
\usepackage{tikz-cd}

\DeclareMathOperator{\id}{id}

\newcommand{\ZZ}{\mathbb{Z}}

\newcommand{\CC}{\mathbb{C}}

\newcommand{\cpr}{\vee}
\newcommand{\cpro}{\vee_{\Omega}}

\makeatletter
\newcommand{\ostar}{\mathbin{\mathpalette\make@circled\ast}}
\newcommand{\make@circled}[2]{%
  \ooalign{$\m@th#1\smallbigcirc{#1}$\cr\hidewidth$\m@th#1#2$\hidewidth\cr}%
}
\newcommand{\smallbigcirc}[1]{%
  \vcenter{\hbox{\scalebox{0.77778}{$\m@th#1\bigcirc$}}}%
}
\makeatother

\renewcommand{\setminus}{\smallsetminus}

\renewcommand{\gg}{\mathfrak{g}}
\newcommand{\hh}{\mathfrak{h}}

\theoremstyle{plain}
\newtheorem{theorem}{Theorem}[section]
\newtheorem{prop}[theorem]{Proposition}
\newtheorem{lemma}[theorem]{Lemma}
\newtheorem{cor}[theorem]{Corollary}  

\newtheorem*{thm}{Theorem}

\theoremstyle{definition}
\newtheorem{definition}[theorem]{Definition}

\theoremstyle{remark}
\newtheorem{remark}[theorem]{Remark}

\numberwithin{equation}{section}

\begin{document}
\setlength{\parindent}{0.2cm}

\title{On the String Topology Coproduct for Lie Groups}

\author{Maximilian Stegemeyer}
\address{Max Planck Institute for Mathematics in the Sciences, Inselstra\ss{}e 22, 04103 Leipzig, Germany} 
\address{Mathematisches Institut, Universit\"at Leipzig, Augustusplatz 10, 04109 Leipzig, Germany}
\email{maximilian.stegemeyer@mis.mpg.de}
\date{\today}
\maketitle

\begin{abstract}
The free loop space of a Lie group is homeomorphic to the product of the Lie group itself and its based loop space.
We show that the coproduct on the homology of the free loop space that was introduced by Goresky and Hingston splits into the diagonal map on the group and a based coproduct on the homology of the based loop space.
This result implies that the coproduct is trivial for even-dimensional Lie groups.
Using results by Bott and Samelson, we show that the coproduct is trivial as well for a large family of simply connected Lie groups.
\end{abstract}

\tableofcontents

\section{Introduction}
In the last two decades the study of string topology structures on the homology or cohomology of the free loop space has received a great deal of attention.
The operation that has been studied the most is certainly the Chas-Sullivan product which was introduced by Chas and Sullivan in \cite{chas:1999}.
For an oriented closed manifold $M$ this is a product on the singular homology of the free loop space $\Lambda M$ of the form
$$ \wedge: \mathrm{H}_{i}(\Lambda M)\otimes \mathrm{H}_j(\Lambda M) \to \mathrm{H}_{i+j-n}(\Lambda M)   $$
where $n$ is the dimension of $M$.
In \cite{cohen:2008}, Cohen, Klein and Sullivan show that the Chas-Sullivan product is a homotopy invariant.

Going back to ideas by Sullivan \cite{sullivan:2004}, Goresky and Hingston \cite{goresky:2009} define a coproduct on the homology of the free loop space of an oriented closed $n$-dimensional manifold $M$ relative to the constant loops which takes the form
$$ \cpr : \mathrm{H}_i(\Lambda M, M) \to \mathrm{H}_{i+1-n}(\Lambda M\times \Lambda M,\Lambda M\times M\cup M\times \Lambda M) .   $$
As a dual operation, a product on the cohomology of the free loop space can be defined (see \cite[Section 9]{goresky:2009}).
This homology coproduct and cohomology product have been further studied by Hingston and Wahl  in \cite{hingston:2017}.
The homology coproduct is not a homotopy invariant as an example of Naef \cite{naef:2021} shows.
However, Hingston and Wahl show in \cite{hingston:2019} that the homology coproduct is invariant under homotopy equivalences that satisfy an additional assumption.

It is not surprising that explicit computations with these string topology operations can become very complicated.
There are some results, where the Chas-Sullivan product was computed explicitly.
Using integer coefficients the Chas-Sullivan product has been computed for spheres and complex projective spaces (see \cite{cohen:2003}) and complex Stiefel manifolds (see \cite{tamanoi:2006}).
With rational coefficients, the Chas-Sullivan product has furthermore been computed for complex and quaternionic projective spaces (see \cite{yang:2013}).
Hepworth has computed the Chas-Sullivan product with rational coefficients for the special orthogonal groups (see \cite{hepworth:2010}).

The Goresky-Hingston product in cohomology and the corresponding coproduct in homology have only been computed explicitly in a few instances.
In \cite{goresky:2009} the authors compute the Goresky-Hingston product for spheres of dimension greater than or equal to $3$. 
In \cite{hingston:2017} the homology coproduct is computed for odd-dimensional spheres.

The goal of this article is to study the homology coproduct for Lie groups.
The free loop space of a compact Lie group $G$ splits into the product
$$ \Lambda G\cong G\times \Omega_e G    $$
where $\Omega_e G$ is the space of loops based at the neutral element $e$.
Since the homology of $\Omega_e G$ with coefficients in a commutative ring $R$ is free (see \cite{bott:1956}), the homology of $\Lambda G$ is isomorphic to the tensor product $\mathrm{H}_{\bullet}(G)\otimes \mathrm{H}_{\bullet}(\Omega_e G)$.
The following main result of this article shows that the homology coproduct behaves well under this isomorphism.
This result should be thought of as an analogue to \cite[Theorem 1.1]{hepworth:2010}. There, Hepworth shows that for a Lie group $G$, the Chas-Sullivan ring is isomorphic to the tensor product of the intersection ring on the Lie group $G$ and the Pontryagin ring on the based loop space $\Omega_e G$.

\begin{thm}[Theorem \ref{theorem_splitting}]
Let $G$ be a compact Lie group of dimension $n$ and consider homology with coefficients in a commutative ring $R$.
Under the isomorphism $\mathrm{H}_{\bullet}(\Lambda G, G)\cong \mathrm{H}_{\bullet}(G)\otimes \mathrm{H}_{\bullet}(\Omega_e G,e)$ the coproduct $\cpr$ can be expressed by the tensor product of the map $d_*$ that is induced by the diagonal map $d: G\to G\times G$ and the based coproduct $\cpro$ up to a sign-correction.
More precisely, the following diagram commutes
$$
\begin{tikzcd}
   \bigoplus_{k+j=i} \mathrm{H}_{k}(G)\otimes \mathrm{H}_j(\Omega,e)  \arrow[d, "(-1)^{ni}d_*\otimes \cpro" ] \arrow[]{r}{\cong} &
    \mathrm{H}_{i}(\Lambda,G) \arrow[]{d}{\cpr}
    \\
    \bigoplus_{k+j=i} \mathrm{H}_{k}(G^2)\otimes \mathrm{H}_{j+1-n}(\Omega^2 ,\Omega\times e\cup e\times \Omega) \arrow[]{r}{\cong} 
    &
    \mathrm{H}_{i+1-n}(\Lambda^2,\Lambda\times G\cup G \times \Lambda )
    \end{tikzcd}
$$
where $\Omega = \Omega_e G$ and $\Lambda = \Lambda G$.
\end{thm}
With this result, one can prove directly that the coproduct is trivial for even-dimensional Lie groups.
Combining this with explicit cycles in the based loop space of a compact, simply connected Lie group of rank $r\geq 2$ we are able to show the following.

\begin{thm}[Theorem \ref{theorem_even_d} and Theorem \ref{theorem_rank}]Let $G$ be a compact Lie group.
If $G$ is even-dimensional or if $G$ is simply connected and of rank $r\geq 2$, then the homology coproduct is trivial.
\end{thm}
This article is organized as follows.
In Section \ref{sec2}, we review some facts about loop spaces and define the based homology coproduct as well as the free homology coproduct. We also show a compatibility statement between these two coproducts.
The goal of Section \ref{sec3} is to show that the splitting of the free loop space of a compact Lie group is respected by the homology coproduct.
We also show that the dual cohomology product behaves nicely under this splitting.
In the brief Section \ref{sec_even}, we conclude that the coproduct is trivial for even-dimensional Lie groups.
Finally, in Section \ref{sec_rank} we define explicit cycles of the based loop space of a Lie group that were first introduced by Bott and Samelson \cite{bott:1958a} to prove that the free homology coproduct vanishes for simply connected compact Lie groups of rank $r\geq 2$.

\emph{In this article all manifolds, are assumed to be smooth and connected and all Riemannian metrics are assumed to be smooth. In particular, all Lie groups are assumed to be connected.}

\section{Free and Based Coproduct} \label{sec2}
In this section, we introduce the based coproduct and the free coproduct.
We will then show that these two coproducts are compatible.
Furthermore, the dual cohomology products are introduced.

In the following, let $M$ be an oriented closed $n$-dimensional Riemannian manifold.
We consider absolutely continuous curves in $M$ (see \cite[Definition 2.3.1]{klingenberg:1995}).
Let $$  PM = \big\{ \gamma : I\to M\,|\, \gamma \,\,\text{absolutely continuous}, \,\, \int_0^1 |\Dot{\gamma}(t)|^2 \,\mathrm{d}t < \infty \big\}   $$
be the set of absolutely continuous curves in $M$ with square integrable derivative, where $I = [0,1]$ is the unit interval.
This set can be given a topology and a differentiable structure that make it a Hilbert manifold (see \cite[Theorem 2.3.12]{klingenberg:1995}). 
Note that $PM$ with this topology is homotopy equivalent to the space
$$   C^0 (I, M) = \{ \gamma : I\to M\,|\, \gamma \,\,\,\text{continuous}\}    $$
of continuous paths in $M$ with the compact-open topology (see \cite[Theorem 1.2.10]{klingenberg:78}).

We consider the following submanifolds of $PM$:
The free loop space of $M$ is defined to be
$$  \Lambda M = \{\gamma \in PM \,|\, \gamma(0 ) = \gamma(1)\}    $$
and for a fixed point $p_0\in M$, the based loop space of $M$ in $p_0$ is
$$  \Omega_{p_0}M = \{ \gamma \in PM \,|\,  \gamma(0) = \gamma(1) = p_0\}    .   $$
If it is clear what the basepoint of $\Omega_{p_0}M$ is, we may suppress the index $p_0$ from the notation.
If the manifold in question is clear from the context, we will also write $\Omega$ and $\Lambda$ for $\Omega_{p_0}M$ and $\Lambda M$, respectively.
Note that the trivial loops in $M$ form a submanifold of $\Lambda M$ which is diffeomorphic to $M$ (see \cite[Proposition 1.4.6]{klingenberg:78}).

On the path space $PM$ we consider the function
 $$  \mathcal{L}: PM \to [0,\infty),\qquad  \mathcal{L}(\gamma) = \sqrt{   \int_0^1 |\Dot{\gamma}(t) |^2 \mathrm{d}t   }  $$
which is the square root of the energy functional and which is well-defined by definition of $PM$. 
The energy functional is a continuous function on $PM$ (see \cite[Theorem 2.3.20]{klingenberg:1995}), hence the function $\mathcal{L}$ is continuous as well.

Assume that we have fixed a basepoint $p_0\in M$ and choose an $\epsilon > 0$ smaller than the injectivity radius of $M$.
In order to define the based coproduct and the free coproduct, we need to make the following preparations. 
For the definition of the free coproduct, we closely follow \cite[Section 1.5] {hingston:2017}.
Fix a commutative ring $R$ and consider homology and cohomology with coefficients in $R$.

First, let $\Delta M$ be the diagonal in $M\times M$. 
The diagonal has a tubular neighborhood in $M\times M$ which can be chosen as
$$   U_M = \{ (p,q)\in M\times M \,|\, \mathrm{d}(p,q) < \epsilon \}    $$
where $\mathrm{d}$ is the distance function on $M\times M$ induced by the Riemannian metric on $M$.
This choice of $U_M$ is made as in \cite[Section 1.3]{hingston:2017}.
Choose $\epsilon_0> 0$ such that $\epsilon_0<\epsilon$ and define
$$  U_{M,\geq \epsilon_0} = \{(p,q)\in U_M\,|\, \mathrm{d}(p,q)\geq \epsilon_0\} .    $$

A tubular neighborhood of the diagonal in $M\times M$ is homeomorphic to the normal bundle of $\Delta M\hookrightarrow M\times M$ which itself is isomorphic to the tangent bundle of $M$.
Consequently, the pair $(U_M,U_{M,\geq\epsilon_0})$ is homeomorphic to the pair $(TM^{<\epsilon},TM^{<\epsilon}_{\geq\epsilon_0})$ 
where $TM^{<\epsilon}$ is the open disk bundle
$$  TM^{<\epsilon} = \{ v\in TM \,|\, |v|<\epsilon\}   $$
and $TM^{<\epsilon}_{\geq\epsilon_0}$ is the fiber bundle
$$  TM^{<\epsilon}_{\geq\epsilon_0} = \{v\in TM^{<\epsilon}\,|\, |v|\geq \epsilon_0\} .   $$
The Thom class in $\mathrm{H}^n(TM,TM\setminus M)$ that is defined by the orientation of $M$ induces a class
in $\mathrm{H}^n(TM^{<\epsilon},TM^{<\epsilon}_{\geq\epsilon_0})$ (see \cite[Section 1.3]{hingston:2017})
and therefore we obtain a class
$$ \tau_M \in \mathrm{H}^n (U_M,U_{M,\geq \epsilon_0}) .   $$

Consider the open ball $$ B_{p_0} = \{ q\in M\,|\, \mathrm{d}(p_0,q) < \epsilon\}  \subseteq M    $$
and the inclusion $\iota : B_{p_0} \hookrightarrow U_M$ given by 
$  \iota(q) = (p_0,q) .   $
If we define
$$    B_{p_0,\geq \epsilon_0} = \{ q\in B_{p_0}\,|\, \mathrm{d}(p_0,q)\geq \epsilon_0\}    $$
then the inclusion $\iota$ can be understood as a map of pairs $\iota : (B_{p_0},B_{p_0,\geq \epsilon_0}) \hookrightarrow (U_M,U_{M,\geq \epsilon_0})$.

Clearly, we have
$$  \mathrm{H}^n(  B_{p_0},B_{p_0,\geq\epsilon_0}) \cong \mathrm{H}^{n}(\mathbb{D}^n,\mathbb{S}^{n-1}) \cong R.  $$
Furthermore, under the identification $U_M \cong TM$, one observes that the inclusion $\iota : B_{p_0}\hookrightarrow U_M$ corresponds to the inclusion of the fiber $T_{p_0}M \hookrightarrow TM$.

Recall that the Thom class $\tau$ of an orientable vector bundle $E\to B$ of rank $k$ has the following property. Let $b\in B$ be a point and denote by $E_b$ the fiber over $b$. If we consider the inclusion $i:(E_b,E_b\setminus\{0\})\hookrightarrow (E,E\setminus B)$ then the pull back $i^*\tau\in \mathrm{H}^k(E_b,E_b\setminus \{0\})$ is the generator induced by the orientation of $E_b$.
Hence, in our particular situation one obtains the following.

\begin{lemma} \label{lemma_generators}
The generator $\tau_{0}$ of $\mathrm{H}^n(B_{p_0},B_{p_0,\geq\epsilon_0})$ which is induced by the orientation of $M$ and the class $\tau_M$ satisfy
$$  \iota^* \tau_M  = \tau_0 .    $$
\end{lemma}
Furthermore, define the spaces
\begin{eqnarray*}
  U_{\Lambda} &=& \{ (\gamma,s) \in \Lambda M\times I\,|\, \mathrm{d}(\gamma(0),\gamma(s))< \epsilon\}  \qquad \text{and} \\
  U_{\Lambda,\geq\epsilon_0} & = &
  \{ (\gamma,s)\in U_{\Lambda}\,|\, \mathrm{d}(\gamma(0),\gamma(s))\geq \epsilon_0\}
\end{eqnarray*}
and their based counterparts
\begin{eqnarray*}
  U_{\Omega} &=&  \{(\gamma,s) \in \Omega M \times I\,|\, \mathrm{d}(p_0,\gamma(s)) < \epsilon\} \qquad \text{and}
  \\  U_{\Omega,\geq \epsilon_0} & = & \{(\gamma,s)\in U_{\Omega}\,|\, \mathrm{d}(p_0,\gamma(s)) \geq \epsilon_0\}
  .    
\end{eqnarray*}
Consider the evaluation maps $\mathrm{ev}_{\Lambda}: \Lambda\times I\to M\times M$ and $\mathrm{ev}_{\Omega}: \Omega\times I\to M$ given by $$ \mathrm{ev}_{\Lambda}(\gamma,s) = (\gamma(0),\gamma(s)) \qquad \text{and} \qquad \mathrm{ev}_{\Omega}(\sigma,t) = \sigma(t)$$
for $\gamma\in \Lambda M$, $\sigma\in \Omega_{p_0}M$ and $s,t\in I$.
These maps restrict to maps of pairs
\begin{eqnarray*}
  & &\mathrm{ev}_{\Lambda} : (U_{\Lambda},U_{\Lambda,\geq \epsilon_0})  \to (U_M, U_{M,\geq \epsilon_0})\qquad \text{and}   \\   & & \mathrm{ev}_{\Omega}: (U_{\Omega},U_{\Omega,\geq\epsilon_0}) \to (B_{p_0},B_{p_0,\geq\epsilon_0})  .
\end{eqnarray*}
Therefore we can define the classes
$$  \tau_{\Lambda} = \mathrm{ev}_{\Lambda}^* \tau_M \in \mathrm{H}^n(U_{\Lambda},U_{\Lambda,\geq\epsilon_0}) \qquad \text{and} \qquad \tau_{\Omega} = \mathrm{ev}_{\Omega}^*\tau_0 \in \mathrm{H}^n(U_{\Omega},U_{\Omega,\geq \epsilon_0} ).     $$

\begin{lemma} \label{lemma_thom}
The classes $\tau_{\Lambda}$ and $\tau_{\Omega}$ satisfy $j^*\tau_{\Lambda} = \tau_{\Omega}$ where $j: (U_{\Omega},U_{\Omega,\geq\epsilon_0})\hookrightarrow (U_{\Lambda},U_{\Lambda,\geq\epsilon_0})$ is the inclusion.
\end{lemma}
\begin{proof}
We have
$$  j^*\tau_{\Lambda} = j^* \mathrm{ev}_{\Lambda}^* \tau_M = \mathrm{ev}_{\Omega}^*\iota^*\tau_M = \mathrm{ev}_{\Omega}^* \tau_0 = \tau_{\Omega}     $$
where we used Lemma \ref{lemma_generators} and the fact that $\mathrm{ev}_{\Lambda}\circ j = \iota \circ \mathrm{ev}_{\Omega}$ with $\iota$ as above.
\end{proof}

Consider the spaces
\begin{eqnarray*}
   & & F_{\Lambda} = \{(\gamma,s)\in \Lambda M\times I\,|\,\gamma(s) = \gamma(0)\}   \qquad \text{and} \\ & & F_{\Omega} = \{ (\gamma,s)\in \Omega M\times I\,|\,\gamma(s) = p_0\} .       
\end{eqnarray*}

We now want to define a retraction map $\mathrm{R}_{GH}: U_{\Lambda}\to F_{\Lambda}$.
Let $\gamma\in PM$ be a path and $s\in I$.
In order to make the restriction $\gamma|_{[0,s]}$ a well-defined element of $PM$, we introduce the map $\rho_s : PM \to PM $ given by
$$   \rho_s (\gamma)  (t)    =   \gamma(st)      \qquad \text{for} \,\,\, \gamma \in PM . $$
Similarly, define $\mu_s :PM\to PM$ by
$$ \mu_s (\gamma)(t) = \gamma(s + (1-s)t)\qquad \text{for}\,\,\, \gamma\in PM     $$
The path $\mu_s(\gamma)$ is the restriction of $\gamma$ to the interval $[s,1]$ combined with a reparametrization to make it an element of $PM$.
Furthermore, for a fixed $s\in I$ and two paths $\gamma,\delta\in PM$ with $\gamma(1)=\delta(0)$, we denote by $\mathrm{concat}_s(\gamma,\delta)$ the concatenation of these two paths such that the path $\gamma$ is run through during the interval $[0,s]$ and the path $\delta$ is run through during the interval $[s,1]$.

Given two paths $\gamma,\delta \in PM$ such that $\gamma(1)=\delta(0)$, we define their \textit{optimal concatenation} to be the path
$$  \gamma * \delta =   \mathrm{concat}_{t_*}(\gamma,\delta) ,\qquad \text{where}\,\,\,  t_* = \frac{\mathcal{L}(\gamma)}{\mathcal{L}(\gamma)+\mathcal{L}(\delta)}   $$
if one of $\gamma$ or $\delta$ has positive length.
If $\mathcal{L}(\gamma)=\mathcal{L}(\delta) = 0$, then $\gamma=\delta = p_0$ is a trivial loop at some point $p_0$ in $M$ and we set $\gamma *\delta = p_0$.
For two points $p,q\in M$ with $\mathrm{d}(p,q) < \epsilon$, the path $\overline{pq}\in PM$ is chosen to be the unique distance-minimizing geodesic connecting these two points parametrized on the unit interval $I$.

Using these definitions, we define $\mathrm{R}_{GH}:U_{\Lambda}\to F_{\Lambda}$ by
$$  \mathrm{R}_{GH} (\gamma,s) = \big( \mathrm{concat}_s \big( \rho_s(\gamma) * \overline{\gamma(s)\gamma(0)} \,, \,\,\overline{\gamma(0)\gamma(s)}* \mu_s(\gamma) \big)\,,\,s\big)     $$
for $(\gamma,s)\in U_{\Lambda}$. One checks that this indeed a continuous map.

If we restrict $\mathrm{R}_{GH}$ to $U_{\Omega}\subseteq U_{\Lambda}$, we obtain a map $\mathrm{R}_{GH}:U_{\Omega}\to F_{\Omega}$. Finally, let $\mathrm{cut}: F_{\Lambda}\to \Lambda M\times \Lambda M$ be the map
$$ \mathrm{cut}(\gamma,s) = (\rho_s(\gamma),\mu_s(\gamma)) \qquad \text{for} \,\,\, (\gamma,s)\in F_{\Lambda} . $$
This restricts to a map $\mathrm{cut}: F_{\Omega} \to \Omega_{p_0}M \times \Omega_{p_0}M $.

Let $[I]$ be a generator of $\mathrm{H}_1(I,\partial I)$.
If $p_0\in M$ is the base point of $M$, then by a slight abuse of notation, we shall write $p_0$ for the set $\{p_0\}\subseteq \Omega M$ that consists of the trivial loop at $p_0$.

\begin{definition}
Let $M$ be an oriented closed $n$-dimensional Riemannian manifold with base point $p_0\in M$.
\begin{enumerate}
    \item The \textit{based homology coproduct} is defined as
    the composition
    \begin{eqnarray*} 
     \cpro : \mathrm{H}_{\bullet}(\Omega,p_0) &\xrightarrow{\times [I]}& \mathrm{H}_{\bullet+1}(\Omega\times I,\Omega\times\partial I \cup p_0\times I)
     \\ &\xrightarrow{\tau_{\Omega}\cap }& \mathrm{H}_{\bullet+1-n}(U_{\Omega} , \Omega\times\partial I  \cup p_0\times I) \\
    &\xrightarrow{(\mathrm{R}_{GH})_*}&
    \mathrm{H}_{\bullet+1-n}(F_{\Omega}  , \Omega\times\partial I \cup p_0\times I ) \\
    &\xrightarrow{(\mathrm{cut})_*}& \mathrm{H}_{\bullet +1-n}(\Omega\times\Omega  , \Omega\times p_0\cup p_0\times \Omega). 
\end{eqnarray*}
\item The \textit{free homology coproduct} is defined as the composition
 \begin{eqnarray*} 
     \cpr: \mathrm{H}_{\bullet}(\Lambda,M) &\xrightarrow{\times [I]}& \mathrm{H}_{\bullet+1}(\Lambda\times I, \Lambda\times\partial I\cup M\times I)
     \\ &\xrightarrow{\tau_{\Lambda}\cap }& \mathrm{H}_{\bullet+1-n}(U_{\Lambda}, \Lambda\times\partial I\cup M\times I) \\
    &\xrightarrow{(\mathrm{R}_{GH})_*}&
    \mathrm{H}_{\bullet+1-n}(F_{\Lambda}, \Lambda\times\partial I\cup M\times I) \\
    &\xrightarrow{(\mathrm{cut})_*}& \mathrm{H}_{\bullet +1-n}(\Lambda\times\Lambda, \Lambda\times M\cup M\times \Lambda). 
\end{eqnarray*}
\end{enumerate}
\end{definition}

Note that the relative cap product that we use, requires some care. For example, in the case of the based coproduct, it is understood as a map
$$   \mathrm{H}^n(U_{\Omega},U_{\Omega,\geq \epsilon_0}) \otimes \mathrm{H}_{i}(\Omega\times I,\Omega\times\partial I \cup p_0\times I) \to    \mathrm{H}_{i-n}(U_{\Omega}, \Omega\times\partial I\cup p_0\times I)  $$
and analogously for the free coproduct.
See \cite[Appendix A]{hingston:2017} for details of this construction and a naturality statement.

We now want to show that the based coproduct and the free coproduct are compatible.

\begin{prop} \label{prop_comp}
Let $M$ be an oriented closed Riemannian manifold with base point $p_0\in M$ and consider singular homology with coefficients in a commutative ring $R$. 
The based coproduct and the free coproduct are compatible in the sense that the diagram 
$$
\begin{tikzcd}
\mathrm{H}_k(\Omega, p_0)    \arrow{r}{i_*}\arrow{d}[swap]{\cpro } & \mathrm{H}_k(\Lambda, M)  \arrow{d}{\cpr} \\
 \mathrm{H}_{k+1-n}(\Omega^2,\Omega\times p_0\cup p_0\times\Omega )  \arrow{r}[]{(i,i)_*} & \mathrm{H}_{k+1-n}(\Lambda^2,\Lambda\times M\cup M\times \Lambda)  
\end{tikzcd} $$
commutes. Here $i:\Omega\to \Lambda$ is the inclusion.
\end{prop}
\begin{proof}
The claimed commutativity of the diagram in the statement of the proposition follows if we verify that all subdiagrams of the following diagram commute
\begin{equation*}
\begin{tikzcd}
\mathrm{H}_{k}(\Omega,p_0) \arrow{r}{i_*} \arrow[swap]{d}{\times [I]}
 & 
 \mathrm{H}_{k}( \Lambda,M) \arrow{d}{\times [I]}
 \\
\mathrm{H}_{k+1}(\Omega\times I,\Omega\times\partial I\cup p_0\times I) \arrow{r}{(i,\mathrm{id}_I)_*}\arrow[d,"\tau_{\Omega}\cap" swap]
& 
\mathrm{H}_{k +1}(\Lambda\times I,\Lambda\times \partial I\cup M\times I) \arrow[d,"\tau_{\Lambda}\cap"]
\\ 
\mathrm{H}_{k+1-n}(U_{\Omega},\Omega\times\partial I\cup p_0\times I) \arrow{r}{(i,\mathrm{id}_I)_*} \arrow[d,"(\mathrm{cut}\circ \mathrm{R}_{GH})_*",swap]
& 
\mathrm{H}_{k +1-n}(U_{\Lambda},\Lambda\times \partial I\cup M\times I) \arrow[d,"(\mathrm{cut}\circ\mathrm{R}_{GH})_*" ]
\\ 
\mathrm{H}_{k+1-n}(\Omega \times\Omega,\Omega\times p_0\cup p_0\times \Omega) \arrow{r}{(i,i)_*\,\,}
& \mathrm{H}_{k + 1-n}(\Lambda\times\Lambda,\Lambda\times M\cup M\times \Lambda).
\end{tikzcd}
\end{equation*}
It is clear that the first and the last square of the above diagram commute.
For the middle square, let $X\in \mathrm{H}_{\bullet}(\Omega\times I,\Omega\times\partial I\cup p_0\times I)$.
Then if we write $j$ for $(i,\mathrm{id}_I)$, we see that
$$   j_* (\tau_{\Omega}\cap X) = j_* ( j^*\tau_{\Lambda} \cap X) = \tau_{\Lambda}\cap (j_*X)    $$
by Lemma \ref{lemma_thom} and the naturality of the cap product.
Hence, the above diagram commutes and therefore the coproducts are compatible.
\end{proof}

We will now focus on field coefficients for the rest of this section.
If $\mathbb{F}$ is a field and if the homology of $\Omega$, resp. $\Lambda$ is of finite type, then the based and free coproduct induce products in cohomology.
Note that the following product on the cohomology of the free loop space was first introduced by Goresky and Hingston in \cite{goresky:2009}.

\begin{definition}\label{def_coho_product}
Let $M$ be an oriented, closed manifold with base point $p_0\in M$ and let $\mathbb{F}$ be a field.
Assume that the homology of $\Omega M$, resp. $\Lambda M$ is of finite type.
\begin{enumerate}
    \item If $\alpha\in \mathrm{H}^i(\Omega,p_0;\mathbb{F})$, $\beta\in \mathrm{H}^j(\Omega,p_0 ;\mathbb{F})$, then their \textit{based cohomology product} $\alpha\ostar_{\Omega} \beta$ is defined to be the unique class in $\mathrm{H}^{i+j+n-1}(\Omega,p_0;\mathbb{F})$ that satisfies
$$   \langle \alpha \ostar_{\Omega} \beta , X\rangle = \langle \alpha \times \beta , \cpro X\rangle \qquad \text{for all} \,\,\, X\in\mathrm{H}_{\bullet}(\Omega,p_0;\mathbb{F}).    $$
\item If $\alpha\in \mathrm{H}^i(\Lambda,M;\mathbb{F})$, $\beta\in \mathrm{H}^j(\Lambda,M ;\mathbb{F})$, then their \textit{free cohomology product} $\alpha\ostar \beta$ is defined to be the unique class in $\mathrm{H}^{i+j+n-1}(\Lambda,M ;\mathbb{F})$ that satisfies
$$   \langle \alpha \ostar \beta , X\rangle = \langle \alpha \times \beta , \cpr X\rangle \qquad \text{for all} \,\,\, X\in\mathrm{H}_{\bullet}(\Lambda,M;\mathbb{F}).    $$
\end{enumerate}
\end{definition}

\begin{remark}
Let $M$ be an oriented closed Riemannian manifold.
\begin{enumerate}
    \item The free homology coproduct was first introduced in \cite[Section 8]{goresky:2009}.
    In the above exposition we followed the definition of the coproduct in \cite{hingston:2017}, where the authors give a chain-level definition. They then show that the induced coproduct in homology is equivalent to the definition in \cite{goresky:2009}.
    The free cohomology product is more often referred to as the \textit{Goresky-Hingston product}, since it was introduced by Goresky and Hingston in \cite[Section 9]{goresky:2009}.
    Using the chain level definition of the free coproduct, Hingston and Wahl obtain an equivalent definition in \cite{hingston:2017}.
    In order to make a better distinction to the based cohomology product, we will always refer to the Goresky-Hingston product as the free cohomology product.
    \item Note that we used the names product and coproduct.
    However, the operations only have the algebraic properties of what is usually understood as product and coproduct if we introduce additional sign conventions (see \cite{hingston:2017}).
    Since this paper mostly considers situations in which the products and coproducts are trivial, we will not deal with the signs and stick to the above definitions.
    \item One can also extend all the coproducts and products to absolute homology and cohomology, respectively.
    As an example, consider the free homology coproduct
    $$   \cpr : \mathrm{H}_i(\Lambda ,M)\to \mathrm{H}_{i+1-n}(\Lambda\times \Lambda,\Lambda \times M\cup M \times \Lambda) . $$
    The homology of $\Lambda $ with coefficients in a commutative ring $R$ is isomorphic to the direct sum
    $$  \mathrm{H}_i(\Lambda ) \cong \mathrm{H}_i(\Lambda, M)\oplus \mathrm{H}_i(M)   . $$
    The idea of extending the free homology coproduct to a map
    $$  \hat{\vee} :\mathrm{H}_i(\Lambda) \to \mathrm{H}_{i+1-n}(\Lambda\times \Lambda)          $$
    is therefore to extend the relative homology coproduct $\cpr$ on $\mathrm{H}_i(\Lambda, M)$ by the trivial map on $\mathrm{H}_i(M)$ in order to get a well-defined map on $\mathrm{H}_i(\Lambda)$.
    Then, one has to lift the resulting class in $\mathrm{H}_{i+1-n}(\Lambda^2,\Lambda\times M \cup M\times \Lambda)$ to a class in $\mathrm{H}_{i+1-n}(\Lambda^2)$.
    See \cite{hingston:2017} for details.
    While it might indeed be more convenient to work on the absolute homology of the free loop space, these extensions do not bring any new topological information into play.
    Therefore, for the rest of this article, we will stick to the relative definitions.
\end{enumerate}
\end{remark}


\section{Splitting of the Coproduct} \label{sec3}

In this section, we will discuss how the based coproduct and the free coproduct are related for Lie groups.
While for arbitrary manifolds, we only have the compatibility statement of Proposition \ref{prop_comp}, we will now show that in the case of a Lie group the free coproduct is completely determined by the based coproduct and the push-forward of the diagonal  map of the group itself.
After that, we will examine how this result transfers to the cohomology products.

Let $G$ be a compact Lie group.
As base point we will always choose the unit element $e\in G$.
There is a homeomorphism $$\Phi: G\times \Omega G\to \Lambda G , \qquad  \Phi(g,\gamma) = (t\mapsto g\gamma(t)) .    $$
Observe that its inverse is the map 
\begin{equation} \label{eq_Psi}
 \Psi: \Lambda G\to G\times \Omega G,\qquad       \Psi (\gamma) = (\gamma(0), t\mapsto (\gamma(0)^{-1}\gamma(t) ))     .
\end{equation}

Bott \cite{bott:1956} has shown that the integer homology of $\Omega G$ is free and that it is non-trivial only in even degrees.
By the universal coefficient theorem, this property then holds for homology with coefficients in an arbitrary commutative ring $R$ as well.
The same also holds for the relative homology $\mathrm{H}_{\bullet}( \Omega G,e)$.
Consequently, there is a Künneth isomorphism
$$  \mathrm{H}_i (G\times \Omega G, G\times e) \cong \bigoplus_{k+j=i} \mathrm{H}_k(G)\otimes \mathrm{H}_j(\Omega G, e) .$$
Combining this with the homeomorphism $\Phi$ we get an isomorphism
$$    \Theta_* :  (\mathrm{H}_{\bullet}(G)\otimes \mathrm{H}_{\bullet}(\Omega G,e))_i  \xrightarrow[]{\times }  \mathrm{H}_i (G\times \Omega G,G\times e) \xrightarrow[]{\Phi_*} \mathrm{H}_i(\Lambda G, G) .   $$
Hepworth \cite{hepworth:2010} has shown that the isomorphism of groups $\Theta_*$ becomes a ring isomorphism if we equip $\mathrm{H}_{\bullet}(G)$ with the intersection product, $\mathrm{H}_{\bullet}(\Omega G)$ with the Pontryagin product and $\mathrm{H}_{\bullet}(\Lambda G)$ with the Chas-Sullivan product.
The following result can be seen as an analogous result for the coproduct.

For the statement of the theorem, observe that there is a homeomorphism $$\widetilde{\Phi}:G\times G\times \Omega G\times \Omega G \to \Lambda G\times \Lambda G$$ given by
\begin{equation} \label{eq_phitilde}
       \widetilde{\Phi}(g_1,g_2,\gamma_1,\gamma_2) = ((t\mapsto g_1\gamma_1(t)),(s\mapsto g_2\gamma_2(s)) ).   
\end{equation}

Since the relative homology $\mathrm{H}_{\bullet}(\Omega G,e)$ is free and concentrated in even degrees, there is a Künneth isomorphism
$$ \mathrm{H}_{i}(\Omega G\times \Omega G,\Omega G\times e\cup e\times \Omega G) 
\cong
\bigoplus_{k+j=i}   \mathrm{H}_{k}(\Omega G,e)\otimes \mathrm{H}_{j}(\Omega G,e) \,\, ,  
$$
see \cite[Theorem 5.3.10]{spanier:1995} for a relative version of the Künneth theorem which applies to this situation.
Hence, the relative homology $\mathrm{H}_{\bullet}(\Omega G\times \Omega G,\Omega G\times e\cup e\times \Omega G)$ is also free and concentrated in even degrees.

Consequently, there is again a Künneth isomorphism
$$ \mathrm{H}_{i}(G^2\times\Omega^2, G^2\times ( \Omega \times e \cup e\times  \Omega)) 
\cong
\bigoplus_{k+j=i} \mathrm{H}_k(G^2) \otimes \mathrm{H}_j(\Omega ^2,\Omega\times e\cup e\times \Omega).   $$
This yields an isomorphism
\begin{eqnarray*}
     \widetilde{\Theta}_* :  \big(\mathrm{H}_{\bullet}(G^2)\otimes \mathrm{H}_{\bullet}(\Omega^2,\Omega\times e\cup e\times \Omega)\big)_i  &\xrightarrow[]{\times } & \mathrm{H}_i (G^2 \times \Omega^2, G^2\times(\Omega\times e\cup e \times \Omega )) \\ &\xrightarrow[]{\widetilde{\Phi}_*}& \mathrm{H}_i(\Lambda^2 , \Lambda\times G \cup G\times \Lambda) .   
\end{eqnarray*}

\begin{theorem} \label{theorem_splitting}
Let $G$ be a compact Lie group of dimension $n$.
Consider homology with coefficients in a commutative ring $R$.
Under the isomorphisms $\Theta_*$ and $\widetilde{\Theta}_*$ the free coproduct $\cpr$ can be expressed by the tensor product of the map $d_*$ that is induced by the diagonal map $d: G\to G\times G$ and the based coproduct $\cpro$ up to a sign-correction.
More precisely, for all $i\in \mathbb{N}_0$ the following diagram commutes
$$
\begin{tikzcd}
   \bigoplus_{k+j=i} \mathrm{H}_{k}(G)\otimes \mathrm{H}_j(\Omega,e)  \arrow[d, "(-1)^{ni}d_*\otimes \cpro" ] \arrow[]{r}{\Theta_*} &
    \mathrm{H}_{i}(\Lambda,G) \arrow[]{d}{\cpr}
    \\
    \bigoplus_{k+j=i} \mathrm{H}_{k}(G^2)\otimes \mathrm{H}_{j+1-n}(\Omega^2,\Omega\times e\cup e\times \Omega) \arrow[]{r}{\widetilde{\Theta}_*} 
    &
    \mathrm{H}_{i+1-n}(\Lambda^2,\Lambda\times G\cup G\times \Lambda)
    \end{tikzcd}
$$
\end{theorem}
\begin{proof}
For the proof, we fix a left-invariant metric on $G$.
Let $k,j\in\mathbb{N}$ and put $l = k+ j + 1-n$.
Let $\mathrm{pr}:G\times \Omega\times I \to \Omega\times I$ be the projection onto the last two factors.

We begin by proving that the following diagram commutes
$$
\begin{tikzcd}
   \mathrm{H}_{k}(G)\otimes \mathrm{H}_j(\Omega,e)  \arrow[]{d}{\times}\arrow[]{r}{\Theta_*} &
    \mathrm{H}_{k+j}(\Lambda,G) \arrow[]{d}{=}
    \\
    \mathrm{H}_{k+j}(G\times\Omega,G\times e) \arrow[]{d}{\times [I]} \arrow[]{r}{\Phi_*}
    & \mathrm{H}_{k+j}(\Lambda,G) \arrow[]{d}{\times [I]}
    \\
    \mathrm{H}_{k+j+1}(G\times \Omega\times I,G\times (e\times I\cup  \Omega\times \partial I ) ) \arrow[]{r}{(\Phi,\id_I)_*} \arrow[]{d}{\mathrm{pr}^*\tau_{\Omega}\cap }   
    &
    \mathrm{H}_{k+j+1}(\Lambda\times I,G\times I\cup \Lambda\times \partial I) \arrow[]{d}{\tau_{\Lambda}\cap}
    \\
    \mathrm{H}_{l}(G\times U_{\Omega},G\times( e\times I \cup \Omega\times \partial I )  ) \arrow[]{r}{(\Phi,\id_I)_*} \arrow[]{d}{(\id_G,\mathrm{R}_{GH})_*}
    &
    \mathrm{H}_l(U_{\Lambda},G\times I\cup \Lambda\times \partial I) \arrow[]{d}{(\mathrm{R}_{GH})_*}
    \\ 
    \mathrm{H}_l(G\times F_{\Omega}, G\times  ( e \times I\cup  \Omega\times \partial I )  ) \arrow[]{r}{(\Phi,\id_I)_*} 
    \arrow[]{d}{(d,\mathrm{cut})_*}
    &
    \mathrm{H}_l(F_{\Lambda}, G\times I\cup \Lambda\times \partial I) \arrow[]{d}{\mathrm{cut}_*}
    \\
    \mathrm{H}_l(G^2\times \Omega^2, G^2\times(\Omega\times e \cup e\times\Omega)) \arrow[]{r}{\widetilde{\Phi}_*}  
    & \mathrm{H}_l(\Lambda^2,\Lambda\times G\cup G\times \Lambda)
\end{tikzcd}
$$

The commutativity of the first two squares is clear.
Consider the third square
$$   
\begin{tikzcd}
   \mathrm{H}_{k+j+1}(G\times \Omega\times I,G\times (e\times I\cup  \Omega\times \partial I ) ) \arrow[]{r}{(\Phi,\id_I)_*} \arrow[]{d}{\mathrm{pr}^*\tau_{\Omega}\cap }   
    &
    \mathrm{H}_{k+j+1}(\Lambda\times I,G\times I\cup \Lambda\times \partial I) \arrow[]{d}{\tau_{\Lambda}\cap}
    \\
    \mathrm{H}_{l}(G\times U_{\Omega},G\times ( e \times I \cup  \Omega\times \partial I ) ) \arrow[]{r}{(\Phi,\id_I)_*} 
    &
    \mathrm{H}_l(U_{\Lambda},G\times I\cup \Lambda\times \partial I)
\end{tikzcd}
$$
To prove that this diagram commutes, let $X\in \mathrm{H}_{\bullet}(G\times \Omega\times I,G\times (e\times I\cup  \Omega\times\partial I)  )$. We need to show that
\begin{equation} \label{eq_thom_lie_1}
       \tau_{\Lambda}\cap ((\Phi,\id_I)_* X) = (\Phi,\id_I)_*(\mathrm{pr}^*\tau_{\Omega} \cap X).    
\end{equation}
By naturality, the left hand side of this equation is
\begin{equation} \label{eq_thom_lie_2}
      \tau_{\Lambda}\cap((\Phi,\id_I)_*X) = (\Phi,\id_I)_*( (\Phi,\id_I)^*\tau_{\Lambda} \cap X)   
\end{equation}
so it suffices to show that $\mathrm{pr}^*\tau_{\Omega} = (\Phi,\id_I)^*\tau_{\Lambda}$.

Consider the tubular neighborhood
$$ U_G = \{ (p,q)\in G\times G \,|\, \mathrm{d}(p,q)<\epsilon\}   $$
where $\mathrm{d}$ is the distance-function induced by the chosen Riemannian metric.
The tubular neighborhood $U_G$ is globally trivial in the sense that the map $\chi \colon U_G \to G\times B_e$ given by
$$  \chi (g,h) = (g,g^{-1}h) \qquad\text{for} \qquad (g,h)\in U_G    $$
is a homeomorphism. 
Here we use that the metric on $G$ is left-invariant, therefore also the distance function is $G$-invariant.

It follows that the Thom class $\tau_G \in \mathrm{H}^n(U_G,U_{G,\geq \epsilon_0})$ can be written as
$$   \tau_G = \chi^* (1_G \times \tau_{e})    $$
where $\tau_{e} \in \mathrm{H}^n(B_e,B_{e,\geq \epsilon_0})$ is the class defined for the based coproduct as in Lemma \ref{lemma_generators}.
Using this identity, we see that
$$  (\Phi,\id_I)^*\tau_{\Lambda} = [\chi\circ \mathrm{ev}_{\Lambda} \circ(\Phi,\id_I)]^*(1_G\times \tau_{e})  .    $$
It furthermore holds that
$$ (\chi\circ \mathrm{ev}_{\Lambda} \circ (\Phi,\id_I)) (g,\gamma,s) = (g,\gamma(s)) = (\id_G,\mathrm{ev}_{\Omega}) (g,\gamma,s) $$ for $(g,\gamma,s)\in G\times \Omega\times I $ . 
Consequently,
$$  (\Phi,\id_I)^*\tau_{\Lambda} = (\id_G,\mathrm{ev}_{\Omega})^*(1_G\times \tau_{e}) = 1_G\times \mathrm{ev}_{\Omega}^*\tau_{e} = \mathrm{pr}^*\tau_{\Omega} . $$
Consequently, we see from equations \eqref{eq_thom_lie_1} and \eqref{eq_thom_lie_2} that the third square commutes.

For the fourth square 
$$ 
\begin{tikzcd}
    \mathrm{H}_{l}(G\times U_{\Omega},G\times (e\times I \cup  \Omega\times \partial I) ) \arrow[]{r}{(\Phi,\id_I)_*} \arrow[]{d}{(\id_G,\mathrm{R}_{GH})_*}
    &
    \mathrm{H}_l(U_{\Lambda},M\times I\cup \Lambda\times \partial I) \arrow[]{d}{(\mathrm{R}_{GH})_*}
    \\ 
    \mathrm{H}_l(G\times F_{\Omega}, G\times (e \times I\cup \Omega\times \partial I ) ) \arrow[]{r}{(\Phi,\id_I)_*} 
    &
    \mathrm{H}_l(F_{\Lambda}, M\times I\cup \Lambda\times \partial I) 
\end{tikzcd}
$$
we prove its commutativity by showing that the diagram of maps
$$ 
\begin{tikzcd}
    G\times U_{\Omega} \arrow[]{r}{(\Phi,\id_I)} \arrow[swap]{d}{(\id_G,\mathrm{R}_{GH})}
    &
    U_{\Lambda} \arrow[]{d}{\mathrm{R}_{GH}}
    \\ 
    G\times F_{\Omega} \arrow[]{r}{(\Phi,\id_I)} 
    &
   F_{\Lambda}
\end{tikzcd}
$$
commutes.
Let $(g,\gamma,s)\in G\times U_{\Omega}$, i.e. $\mathrm{d}(e, \gamma(s))< \epsilon$.
Using the shorthand notation $\gamma_t = \gamma(t)$ for $t\in[0,1]$, we have
\begin{eqnarray*}
((\Phi,\mathrm{id}_I)\circ(\mathrm{id}_G,\mathrm{R}_{GH})) (g,\gamma,s) &=& \big( g\cdot \mathrm{concat}_s\big[ (\rho_s(\gamma)*\overline{\gamma_s\gamma_0}), (\overline{\gamma_0\gamma_s}* g \cdot \mu_s(\gamma)) \big]    , \,s\big)
\\ &=&
\big( \mathrm{concat}_s \big[ g\cdot \rho_s(\gamma) * \overline{(g\gamma_s)( g\gamma_0)},      \overline{(g\gamma_0)(g\gamma_s)} * g\cdot \mu_s(\gamma)           \big]
,s \big)
\\
&=&
(\mathrm{R}_{GH}\circ (\Phi,\mathrm{id}_I)) (g,\gamma,s) .
\end{eqnarray*}
The second equality holds by our use of a left-invariant metric on $G$ which implies that the geodesic segment from $g\gamma_0$ to $g\gamma_s$ can be expressed as
$$   \overline{(g\gamma_0)(g\gamma_s)} = g\cdot \overline{\gamma_0\gamma_s}      $$
since left-multiplication by $g\in G$ is an isometry.
This proves that the fourth square commutes.

The commutativity of the last square can be seen in an analogous manner by checking that the underlying maps commute.
Thus we have shown that the large diagram commutes. 

Call the map that is defined by the left vertical side of the large diagram $$\Xi : \mathrm{H}_k(G)\otimes \mathrm{H}_j(\Omega,e)\to \mathrm{H}_{l}(G^2\times \Omega^2,G^2\times (\Omega\times e\cup e \times \Omega)).$$
In order to complete the proof of the theorem, we need to show that the following diagram commutes 
$$ \begin{tikzcd}
    & \mathrm{H}_k(G)\otimes \mathrm{H}_j(\Omega,e)\arrow[swap]{ld}{(-1)^{n(k+j)}d_*\otimes \cpr_{\Omega} } \arrow[]{dd}{\Xi}  \\
    \mathrm{H}_k(G^2) \otimes \mathrm{H}_{j+1-n}(\Omega^2,\Omega\times e\cup e \times \Omega) \arrow[]{rd}{\times }  &
    \\
   &   \mathrm{H}_{l}(G^2\times \Omega^2,G^2\times(\Omega\times e\cup e \times \Omega ))
\end{tikzcd}   $$
This is easily seen by unwinding the definitions. 
Let $X\in \mathrm{H}_k(G)$ and $Y\in \mathrm{H}_j(\Omega,e)$, then
\begin{eqnarray*}
    \Xi (X\otimes Y)&=&  (d,\mathrm{cut})_*(\id_G,\mathrm{R}_{GH})_* ( \mathrm{pr}^*\tau_{\Omega} \cap (X\times Y\times [I]) ) \\
    & = &  (d,\mathrm{cut}\circ \mathrm{R}_{GH})_* ((1_G\times \tau_{\Omega}) \cap (X \times (Y\times [I]))) \\ &=&
    (-1)^{nk}(d,\mathrm{cut}\circ \mathrm{R}_{GH})_*(X\times (\tau_{\Omega}\cap (Y\times [I]))) 
    \\ &=& (-1)^{nk} (d_*X \times \cpro Y) .
\end{eqnarray*}
Now, if $Y$ is non-trivial, then $j$ is even, hence we see that
$$  \Xi (X\otimes Y) =  (-1)^{n(k+j)} (d_*X \times \cpro Y) .  $$

The commutativity of this last diagram combined with the commutativity of the large diagram at the beginning of the proof show the claim.
\end{proof}

We now prove a dual result about the cohomology products introduced in Section \ref{sec2}.
We will only deal with cohomology with coefficients in a field $\mathbb{F}$.
Note that since the cohomology of the Lie group is of finite type, the cross product in cohomology induces an isomorphism
$$ (\mathrm{H}^{\bullet}(G;\mathbb{F}) \otimes \mathrm{H}^{\bullet}(\Omega G , e;\mathbb{F}))^i \cong \mathrm{H}^i(G\times \Omega G , G\times e;\mathbb{F})     $$
see \cite[Theorem 5.6.1]{spanier:1995}.
Combining this with the pull-back of the map $\Psi$ defined in equation \eqref{eq_Psi}, we get an isomorphism
$$  \Theta^* :    (\mathrm{H}^{\bullet}(G;\mathbb{F}) \otimes \mathrm{H}^{\bullet}(\Omega ,e;\mathbb{F}))^i \xrightarrow[]{\times} \mathrm{H}^i(G\times \Omega , G\times e;\mathbb{F})    \xrightarrow[]{\Psi^*}  \mathrm{H}^i(\Lambda , G;\mathbb{F}).  $$

\begin{theorem} \label{theorem_product}
Let $G$ be a compact Lie group.
Assume that the homology $\mathrm{H}_{\bullet}(\Omega G;\mathbb{F})$ is of finite type.
The free cohomology product on $\Lambda G$ can be expressed by the tensor product of the cup product on the cohomology of $G$ and the based cohomology product on $\Omega G$.
More precisely, if $\alpha\in\mathrm{H}^i(\Lambda, G ;\mathbb{F})$ and $\beta\in \mathrm{H}^j(\Lambda, G;\mathbb{F})$ are of the form
$$  \alpha = \Psi^*(a\times A) \qquad \text{and}\qquad \beta = \Psi^*(b\times B)    $$
where $a,b\in\mathrm{H}^{\bullet}(G;\mathbb{F})$ and $A,B\in\mathrm{H}^{\bullet}(\Omega, e;\mathbb{F})$, then
$$  \alpha\ostar \beta = (-1)^{n(i+j)} \Psi^*[(a\cup b) \times (A\ostar_{\Omega}B)] .     $$
\end{theorem}
\begin{proof}
Since the homology $\mathrm{H}_{\bullet}(\Omega G;\mathbb{F})$ is of finite type, so is $\mathrm{H}_{\bullet}(\Lambda G;\mathbb{F})$ and thus the cohomology products are defined.
The free cohomology product of $\alpha$ and $\beta$ is determined by the natural pairing of homology and cohomology, i.e. we have
\begin{equation} \label{eq_nat_pairing}
       \langle \alpha\ostar \beta ,X\rangle = \langle \alpha\times \beta,\cpr X\rangle \qquad \text{for all} \,\,\,X\in\mathrm{H}_{\bullet}(\Lambda, G;\mathbb{F}).     
\end{equation}
Since the natural pairing is trivial if the homology class and the cohomology class have different degrees, it is sufficient to consider homology classes $X$ of degree $k = i+j+n-1$.

Let $X\in \mathrm{H}_k(\Lambda, G;\mathbb{F})$ and assume that the class $X$ is of the form
\begin{equation} \label{eq_product}
      X = \Phi_*(x\times \xi) \qquad \text{with}\,\,\,\, x\in \mathrm{H}_{\bullet}(G;\mathbb{F}),\, \xi\in\mathrm{H}_{\bullet}(\Omega , e;\mathbb{F}) .  
\end{equation}
Equation \eqref{eq_nat_pairing} then becomes
\begin{eqnarray} \label{eq_alg}
\langle \alpha\times\beta,\cpr X\rangle & = & \langle \Psi^*(a\times A) \times \Psi^*(b\times B), (-1)^{nk} \widetilde{\Phi}_*(d_*x \times \cpro\xi )\, \rangle \nonumber \\
& = & (-1)^{nk} \langle \widetilde{\Phi}^*(\Psi,\Psi)^* (a\times A\times b\times B) , d_*x \times \cpro \xi\rangle \nonumber \\
& = & (-1)^{nk} \langle  ((\Psi,\Psi)\circ \widetilde{\Phi})^* (a\times A\times b\times B) , d_* x \times \cpro \xi\rangle 
\end{eqnarray}
where $\widetilde{\Phi}$ was defined in equation \eqref{eq_phitilde} and
where we used Theorem \ref{theorem_splitting} in the first equality.

A direct computation shows that
$$   ((\Psi,\Psi)\circ\widetilde{\Phi})(g_1,g_2,\gamma_1,\gamma_2) = (g_1,\gamma_1,g_2,\gamma_2) \qquad \text{for} \,\,\, g_1,g_2\in G,\,\gamma_1,\gamma_2\in\Omega    .  $$
Therefore, $(\Psi,\Psi)\circ\widetilde{\Phi}$ is just the swapping map
$$   G\times G\times \Omega \times \Omega \to G\times \Omega \times G \times\Omega    $$
interchanging the second and the third factor.
By the standard properties of the cross product we obtain
$$ ((\Psi,\Psi)\circ \widetilde{\Phi})^* (a\times A\times b\times B)  = (-1)^{|A||b|} a\times b\times A\times B $$
where $|\cdot|$ is the degree of a cohomology class.
As for homology, the cohomology of the based loop space is only non-trivial in even degrees, so if the class $A$ is non-trivial, the sign $(-1)^{|A||b|}$ is always equal to $1$. 
Going back to equation \eqref{eq_alg}, we see that
\begin{eqnarray*}
\langle \alpha\ostar \beta ,X\rangle & = &
(-1)^{nk} \langle a\times b\times A\times B, d_*x\times \cpro \xi\rangle 
\\
&=& (-1)^{nk} \langle a\times b,d_*x\rangle \langle A\times B,\cpro \xi\rangle
\\
&=& (-1)^{nk} \langle (a\cup b) \times (A\ostar_{\Omega} B), x\times \xi\rangle 
\\
&=& (-1)^{nk} \langle \Psi^* ((a\cup b) \times (A\ostar_{\Omega} B) ), X \rangle
\end{eqnarray*}
where we used the definitions of the cup product and of the based cohomology product, respectively.
For the sign, we observe that 
$$  (-1)^{nk} =  (-1)^{n(i+j)}(-1)^{n(n-1)} = (-1)^{n(i+j)}  .    $$
This shows that
\begin{equation} \label{eq_defining}
      \langle \alpha\ostar \beta ,X\rangle =  (-1)^{n(i+j)} \langle \Psi^* ((a\cup b) \times (A\ostar_{\Omega} B) ), X \rangle  
\end{equation}
for all classes $X$ that can be written as in equation \eqref{eq_product}.
But since these classes span all of $\mathrm{H}_k(\Lambda,G;\mathbb{F})$, it follows that equation \eqref{eq_defining} holds for all $X\in\mathrm{H}_k(\Lambda, G;\mathbb{F})$.
By the non-degeneracy of the natural pairing this completes the proof of the theorem.
\end{proof}


\section{Triviality of the Coproduct for Even-Dimensional Lie Groups} \label{sec_even}

We are now going to show that the coproduct is trivial for large classes of Lie groups.
In this section, we will consider compact, even-dimensional Lie groups, where the triviality follows directly from Theorem \ref{theorem_splitting}.

\begin{theorem} \label{theorem_even_d}
Let $G$ be an even-dimensional compact Lie group.
Consider homology with coefficients in a commutative ring $R$.
Then the based coproduct $\cpro$ and the free coproduct $\cpr$ are trivial.
\end{theorem}
\begin{proof}
As we have mentioned earlier, the relative homology of the based loop space $\mathrm{H}_{\bullet}(\Omega G,e)$ is free and concentrated in even degrees (see \cite{bott:1956}).
Consequently, by the Künneth isomorphism the same properties hold for the relative homology $\mathrm{H}_{\bullet}(\Omega G\times \Omega G,\Omega G\times e\cup e\times \Omega G)$.

The based coproduct is a map
$$  \cpro : \mathrm{H}_i(\Omega G,e) \to \mathrm{H}_{i+1-n}(\Omega G\times \Omega G,\Omega G\times e\cup e \times \Omega G)     $$
with degree shift $1-n$.
Since $n$ is even, this degree shift is odd.
Consider now the case where $i$ is even. Then the degree $i+1-n$ is odd, so $\mathrm{H}_{i+1-n}(\Omega G^2,\Omega G\times e\cup e\times \Omega G )$ is trivial.
Hence, for even $i$, the based coproduct must be trivial.
In the other case, one can see similarly that the based coproduct must be trivial as well.
This shows that the based coproduct must be trivial.

By Theorem \ref{theorem_splitting}, the free coproduct $\cpr$ on $\mathrm{H}_{\bullet}(\Lambda, G)$ can be expressed as 
\begin{equation} \label{eq_split_even}
     d_* \otimes \cpro : (\mathrm{H}_{\bullet}(G)\otimes \mathrm{H}_{\bullet}(\Omega G,e))_i \to (\mathrm{H}_{\bullet}(G^2)\otimes \mathrm{H}_{\bullet}(\Omega G^2, \Omega G\times e \cup e \times \Omega G))_{i+1-n}     
\end{equation}
under the isomorphisms described in Theorem \ref{theorem_splitting}.
We can drop the sign that appears in Theorem \ref{theorem_splitting}, since the dimension $n$ of $G$ is even.
As we have already shown, the based coproduct $\cpro$ is trivial in this situation.
This implies that the map in equation \eqref{eq_split_even} is trivial and by Theorem \ref{theorem_splitting} this implies the triviality of the free coproduct.
\end{proof}
\begin{cor}
Let $G$ be a compact even-dimensional Lie group and let $\mathbb{F}$ be a field.
Assume that $\mathrm{H}_{\bullet}(\Omega G;\mathbb{F})$ is of finite type.
Then the based cohomology product and the free cohomology product of $G$ are trivial.
\end{cor}
\begin{proof}
Since $\mathrm{H}_{\bullet}(\Omega G;\mathbb{F})$ is of finite type, so is $\mathrm{H}_{\bullet}(\Lambda G;\mathbb{F})$.
Hence, we can use Theorem \ref{theorem_even_d} to see that the cohomology products must vanish.
\end{proof}

\section{Triviality of the Coproduct for certain Simply Connected Lie Groups} \label{sec_rank}
We now turn to compact, simply connected Lie groups of rank $r\geq 2$.
There are lots of classical results about the topology of the based loop space of a compact Lie group by Bott (see \cite{bott:1956} and \cite{bott:1958b}) and by Bott and Samelson (see \cite{bott:1958a}).
We will use an explicit description of cycles in the based loop space. 
Throughout this section, we consider homology with integer coefficients.

Before, we investigate these cycles, let us define the intersection multiplicity of a homology class.
The following definition is given as in \cite[Definition 5.1]{hingston:2017}.
\begin{definition} Let $M$ be a Riemannian manifold.
     Let $[X]\in \mathrm{H}_{\bullet}(\Lambda M, M)$ be a non-trivial homology class with representing cycle $X\in \mathrm{C}_{\bullet}(\Lambda M, M)$.
     Assume that the relative cycle $X$ is represented by a cycle $x\in \mathrm{C}_{\bullet}(\Lambda M)$.
    The \textit{basepoint intersection multiplicity} $\mathrm{int}([X])$ of the class $[X]$ is the number
    $$    \mathrm{int}([X]) = \inf_{A \sim x} \Big(\sup\big[ \#(\gamma^{-1}(\{\gamma(0)\})) \,|\, \gamma\in \mathrm{Im}(A),\,\,\, \mathcal{L}(\gamma)> 0 \big] \Big)  -1   $$
    where the infimum is taken over all cycles $A \in \mathrm{C}_{\bullet}(\Lambda M)$ homologous to $x$.
\end{definition}
Note that this definition does not depend on the choice of the cycle $x$, since if $x'\in \mathrm{C}_{\bullet}(\Lambda M)$ is another cycle representing the relative cycle $X$, then $x$ and $x'$ differ only by elements of $\mathrm{C}_{\bullet}(M)$. However, in the above definition we do not consider trivial loops, so the basepoint intersection multiplicity is well-defined.

Also note that the correction of $-1$ comes from the fact that we consider loops as maps $\gamma :[0,1]\to X$ with $\gamma(0)=\gamma(1)$. Hence, for any loop, we would count the basepoint twice if we did not perform this correction.
By using finite-dimensional approximations of $\Lambda M$, one can furthermore see that the number $\mathrm{int}([X])$ is always finite (see \cite[p. 44]{hingston:2017}).

Hingston and Wahl show the following result (see \cite[Theorem 3.10 (C)]{hingston:2017}).
\begin{prop} \label{prop_int_triv}
Let $M$ be an oriented, closed manifold.
If $X\in\mathrm{H}_i(\Lambda M, M)$ is a non-trivial homology class with basepoint intersection multiplicity $\mathrm{int}(X) = 1$, then the free coproduct $\cpr X$ vanishes.
\end{prop}

From now on, let $G$ be a compact, simply connected Lie group of rank $r\geq 2$.
In particular, $G$ is then semisimple (see \cite[Section V.7]{broecker:85}).
We shall now describe some explicit cycles in $\Omega G$, closely following the exposition in \cite[Section 3]{bott:1958b}.

Denote the Lie algebra of $G$ by $\gg$.
Let $T\subseteq G$ be a maximal torus, which implies $\mathrm{dim}(T) = r$. Furthermore, let $\gg_{\CC}$ be the complexification of the Lie algebra $\gg$.
Then the Lie algebra $\mathfrak{t}$ of the maximal torus $T$ is a maximal abelian subalgebra of $\gg$ and its complexification $\hh=\mathfrak{t}\oplus i\mathfrak{t}$ is a Cartan subalgebra of $\mathfrak{g}_{\mathbb{C}}$.

A root of the complex semisimple Lie algebra $\gg_{\CC}$ is an element $\alpha$ of the dual space $\hh^*$ such that there is a non-trivial subspace $\gg_{\alpha}\subseteq \gg_{\CC}$ with
$$  [H,X] = \alpha(H) X \qquad \text{for all} \,\,\, H\in\hh,\,\, X\in\gg_{\alpha}.    $$
Denote the set of non-zero roots by $\Delta$, then one obtains a decomposition
$$   \gg_{\CC} = \hh \oplus \bigoplus_{\alpha\in\Delta} \gg_{\alpha} .    $$
If a non-zero root $\alpha\in \Delta$ is restricted to $\mathfrak{t}\subseteq \hh$, then its values are purely imaginary.
\begin{definition}
Let $\alpha\in \Delta$ be a non-zero root and $n\in \ZZ$.
The affine plane
$$  \{H\in\mathfrak{t} \,|\, \alpha(H) = 2\pi i n\}  \subseteq \mathfrak{t}  $$
is called the \textit{singular plane} $(\alpha,n)$.
\end{definition}

If $p=(\alpha,n)$ is a singular plane, we write $\overline{p}$ for its image $\exp(p)\subseteq G$ under the Lie group exponential of $G$.
Let $G(p)$ be the centralizer of $\overline{p}$, i.e. we have
$$  G(p) = \{   g\in G\,|\, gh = hg \,\,\text{for all} \,\, h\in\overline{p}\} .   $$
The group $G(p)$ is a closed subgroup of $G$.

Clearly, the torus $T$ is itself always a closed subgroup of $G(p)$ and one can show that for all singular planes $p$ the dimension of $G(p)$ is strictly larger than $r$ (see e.g. \cite[Lemma VII.4.5]{helgason:78}).

Let $P=(p_1,p_2,\ldots,p_m)$ be an ordered family of singular planes, where $m\in \mathbb{N}$.
Set 
$$  W(P) = G(p_1)\times G(p_2)\times\ldots \times G(p_m)   .   $$
Following \cite{bott:1958b}, there is a right-action $\chi$ of the $m$-fold product of the maximal torus $T$ on $W(P)$ given by
\begin{eqnarray*}    \chi : W(P) \times T^m &\to& W(P)      \\
                ( (k_1,\ldots,k_m), (t_1,\ldots,t_m)) &\mapsto& (k_1 t_1,t_1^{-1} k_2 t_2,\ldots, t_{m-1}^{-1} k_m t_m)  .
\end{eqnarray*}
Note that $T^m$ is a closed subgroup of $W(P)$ but the action above is not the one induced by the group multiplication in $W(P)$.
One can check that the action $\chi$ is proper and free and consequently the quotient
$$ \Gamma_P = W(P) / T^m    $$
is a manifold.
Bott and Samelson show in \cite{bott:1958a} that all $\Gamma_P$ are connected and oriented.

We now construct a continuous map $f_P : \Gamma_P \to \Omega G$.
Let $c = (c_0,c_1,\ldots,c_m)$ be an ordered sequence of polygons in $\mathfrak{t}$ with the following property:
\begin{itemize}
    \item The polygon $c_0$ starts at the origin $0\in\mathfrak{t}$.
    \item For $i=1,\ldots,m$ the endpoint of the polygon $c_{i-1}$ is the start point of the polygon $c_i$ and lies on the singular plane $p_i$.
    \item The polygon $c_m$ ends at the origin $0\in\mathfrak{t}$.
\end{itemize}
For all $i=0,1,\ldots,m$, we can parametrize the polygon $c_i$ on the interval $[\frac{i}{m+1},\frac{i+1}{m+1}]$.
Then we define a map $\widetilde{f}_P : W(P) \to \Omega G$ by setting
$$   \widetilde{f}_P(g_1,g_2,\ldots,g_m) = \mathrm{concat}\Big[  \exp(c_0), g_1\exp(c_1)g_1^{-1}, \ldots,g_1\ldots g_m\exp(c_m)g_m^{-1}\ldots g_1^{-1}    \Big].     $$
By the choice of the polygons, this is a well-defined continuous map. 
One checks that $\widetilde{f}_P$ is invariant under the right-action of $T^m$ by $\chi$, so it descends to a map
$$  f_P : \Gamma_P \to \Omega G .    $$
If we choose an orientation class $[\Gamma_P]\in\mathrm{H}_{\bullet}(\Gamma_P)$ of the manifold $\Gamma_P$, we obtain a homology class
$$  P_* = (f_P)_*[\Gamma_P] \in \mathrm{H}_{\bullet}(\Omega G) .      $$
Bott and Samelson \cite{bott:1958a} have shown that the set
\begin{equation} \label{eq_generating_sing_planes}
  \{ P_*\in\mathrm{H}_{\bullet}(\Omega G)\,|\,P = (p_1,\ldots,p_m) \,\,\,\text{is an ordered family of singular planes}, \,\,\, m\geq 1\}     
\end{equation}
together with a generator of $\mathrm{H}_0(\Omega G)$
generates the homology of $\Omega G$ if $G$ is simply connected.
Hence, the image in $\mathrm{H}_{\bullet}(\Omega G,e)$ of the set in equation \eqref{eq_generating_sing_planes} generates the homology $\mathrm{H}_{\bullet}(\Omega G, e)$.

Note that the map $f_P$ depends on the choice of the polygons $(c_0,\ldots,c_m)$. 
However, assume that we make a different choice, i.e. we choose an ordered sequence of polygons $(d_0,\ldots,d_m)$ that satisfy the same conditions with respect to the family of singular planes $P$.
Then the map $\widehat{f}_P$ which is induced by the polygons $(d_0,\ldots,d_m)$ is homotopic to the map $f_P$,
since we can continuously deform each $d_i$ into $c_i$ and thus obtain the desired homotopy from $\widehat{f}_P$ to $f_P$.

Let $\mathcal{F}$ be the lattice
$$ \mathcal{F} = \{H\in\mathfrak{t}\,|\, \exp(H) = e\}    $$
where $\exp$ is the Lie group exponential of $G$ and $e\in G$ is the unit element.
Recall that the dimension of $\mathfrak{t}$ is equal to the rank $r$.

\begin{definition}
Let $P = (p_1,\ldots,p_m)$ be an ordered family of singular planes.
We say that an ordered family of polygons $(c_0,c_1,\ldots,c_m)$ is \textit{lattice-nonintersecting} if the following holds: No polygon $c_i$ intersects the lattice $\mathcal{F}$ apart from $c_0$ at its start point and $c_m$ at its endpoint.
\end{definition}

In case the rank $r$ of the Lie group $g$ satisfies $r\geq 2$, it is clear that given an ordered family of singular planes $P=(p_1,\ldots,p_m)$ we can choose the polygons $(c_0,c_1,\ldots,c_m)$ to satisfy the conditions with respect to the endpoints of the polygons and to be lattice-nonintersecting at the same time.
This choice is not possible however if  the rank is $r=1$.

 \begin{prop} \label{prop_k_cycl_int}
 Let $G$ be a compact, simply connected Lie group of rank $r\geq 2$.
Then every homology class in $\mathrm{H}_{\bullet}(\Lambda G, G)$ has basepoint intersection multiplicity $1$.
 \end{prop}
 \begin{proof}
By the arguments from Section \ref{sec3} the homology of the free loop space of $G$ relative to the constant loops is isomorphic to the tensor product of the homology of $G$ and of the homology of the based loop space of $G$ relative to the basepoint via the maps
$$   (\mathrm{H}_{\bullet}(G)\otimes \mathrm{H}_{\bullet}(\Omega G,e))_i  \xrightarrow[]{\times }  \mathrm{H}_i (G\times \Omega G,G\times e) \xrightarrow[]{\Phi_*} \mathrm{H}_i(\Lambda G,G)   .   $$
Choose classes $[x_1],\ldots,[x_l] \in\mathrm{H}_{\bullet}(G)$ that generate the homology of $G$.
Then by the considerations before this proposition and the above isomorphism, it is clear that the set
$$  \{ \Phi_* ([x_i]\times P_*)\,|\, i\in\{1,\ldots,l\},\,\,\,P \,\,\text{non-trivial ordered family of singular planes}\}        $$
generates the homology of $\Lambda G$ relative to the trivial loops $G$.

Let $i\in\{1,2,\ldots,l\}$ and let $P$ be a non-trivial ordered family of singular planes.
We want to determine the intersection multiplicity of the homology class $\Phi_*([x_i]\times P_*)$.
Choose an arbitrary representative $x_i$ of $[x_i]$ and a representative $\mathcal{P}$ of $P_*$ that is obtained via a lattice-nonintersecting polygon.
This is possible, since the rank $r$ is greater than or equal to $2$.

A loop $\gamma\in \mathrm{Im}(\Phi_*(x_i \times \mathcal{P} ) )$ is of the form
$$  \gamma = k \cdot \mathrm{concat}\Big[  \exp(c_0), g_1\exp(c_1)g_1^{-1}, \ldots,g_1\ldots g_m\exp(c_m)g_m^{-1}\ldots g_1^{-1}    \Big]  .  $$
Clearly, we have $\gamma(0)=\gamma(1) = k$.
Assume that there is an $s\in(0,1)$ with $\gamma(s) = k$.
Then there is an $i\in\{0,\ldots,m\}$ with $\exp(c_i (s)) = e$ which implies that $c_i(s)\in\mathcal{F}$.
But by the choice of a lattice-nonintersecting polygon, this is a contradiction to $s\in(0,1)$ which proves that $\# (\gamma^{-1}(k)) = 2$.
We conclude that
$$  \mathrm{int}(\Phi_*([x_i]\times P_*)) = 1 .  $$
Since the homology of the free loop space $\Lambda G$ relative to the constant loops $G$ is generated by homology classes of the form that we considered, this shows that every class in $\mathrm{H}_{\bullet}(\Lambda G,G)$ has basepoint intersection multiplicity $1$.
 \end{proof}
 \begin{theorem} \label{theorem_rank}
 Let $G$ be a simply connected, compact Lie group of rank $r \geq 2$.
 Then the free coproduct $\cpr$ is trivial.
 \end{theorem}
 \begin{proof}
 This follows from Propositions \ref{prop_int_triv} and \ref{prop_k_cycl_int}.
 \end{proof}
 \begin{remark}
 \begin{enumerate}
 \item The condition on the rank in Theorem \ref{theorem_rank} cannot be given up.
 The $3$-sphere $\mathbb{S}^3$ can be seen as the Lie group $\mathrm{SU}(2)$ which is a compact, simply connected Lie group of rank $1$. The free coproduct on $\mathbb{S}^3$ is non-trivial (see \cite[Proposition 3.17]{hingston:2017} ).
 \item Classical examples of Lie groups to which Theorem \ref{theorem_rank} applies are the special unitary groups $\mathrm{SU}(n)$ for $n\geq 3$, the Spin groups $\mathrm{Spin}(n)$ for $n\geq 4$ and the compact symplectic groups $\mathrm{Sp}(n)$ for $n\geq 2$.
 \end{enumerate}
 \end{remark}

 \subsection*{Acknowldegements}
 The author wants to thank Philippe Kupper and Stephan Mescher for many helpful comments on this manuscript.

\bibliography{lit}
 \bibliographystyle{amsalpha}
 
\end{document}